\documentclass[a4paper,12pt]{article}

\usepackage{hyperref}
\usepackage{amsmath}
\usepackage{enumerate}
\usepackage{amsthm}
\usepackage{amsfonts}
\usepackage{subeqnarray}
\usepackage{mathrsfs}
\usepackage{latexsym}
\usepackage{dsfont}
\usepackage{mathtools}
\usepackage[margin=2cm]{geometry}
\allowdisplaybreaks[4]

\newcommand{\beq}{\begin{equation}}
\newcommand{\eeq}{\end{equation}}
\newcommand{\ba}{\begin{eqnarray}}
\newcommand{\ea}{\end{eqnarray}}

\newcommand{\inj}{\hookrightarrow}

\newcommand{\beg}{\begin{align}}

\newcommand{\uu}{\mathrm}
\newtheorem{theorem}{Theorem}[section]

\newtheorem{proposition}{Proposition}[section]
\newtheorem{lemma}{Lemma}[section]
\newtheorem{corollary}{Corollary}[section]

\title{A generalized Cauchy-Lipschitz theorem
in low regularity spaces.}
\author{\scriptsize{ARNAUD HEIBIG$^1$}
\thanks{Corresponding author.  
E-mail:  arnaud.heibig@insa-lyon.fr; Fax: +33 472438529}}
\date{}
\numberwithin{equation}{section}
\begin{document}
 \maketitle
\begin{flushleft}
 $^1$ 
 Universit\'e de Lyon, 
Institut Camille Jordan et Insa Lyon, B\^at. Leonard de Vinci 
No. 401, 21 Avenue Jean Capelle, F-69621, Villeurbanne, France.
\end{flushleft}
 
\begin{abstract}
We prove well-posedness for some abstract differential 
equations of the first order. Our result covers the usual case
of Lipschitz composition operators. It also contains the case of some
integro-differential operators acting on spaces with low regularity
indexes. The loss of derivatives induced by such operators has to 
be lower than one, in order to be dominated by the first 
order derivative involved in the problem.
\end{abstract} 
\begin{flushleft}
\end{flushleft}
\section{\large{Introduction.}}
\label{intro}

The aim of this note is to prove an extended Cauchy-Lipschitz
theorem for problems formally written as
$u'=\mathcal{H}_T(u)$ and $u(0) = u_0$. Here, 
$\mathcal{H}_T: U \subset B_{p, q}^s(]0, T[, E)
\rightarrow B_{p, q}^{\sigma}(]0, T[, E)$ is 
a local operator loosing less than one derivative, i.e
$\sigma>s-1$.
 \begin{theorem}\label{terminapok}
Let  
$T> 0$, $R>0$, 
 $1 \leq p, q \leq \infty$,
  $1/p < s< \sigma + 1< 1/p+1$, 
  and $u_0 \in  E$. 
 Assume that  $\mathcal{H}_{T}: 
 B_{p, q}^{s}(]0, T[, E)
 \rightarrow  B_{p, q}^{\sigma}(]0, T[, E)$ is a Lipschitz 
 and local operator.
 Then, there exists $0 < \mathscr{T} \leq T$ such that, 
 for any $0<t_0\leq\mathscr{T}$ 
   the 
 problem: find
 $u\in  B_{p, q}^{s}(]0, t_0[, E)$
 with:
  \begin{equation}\label{seaal}
\begin{cases}
u' = \mathcal{H}_{t_0}
(u)
\\
u(0) = u_0\\
 \end{cases}
\end{equation}
admits exactly one 
distributional solution. This solution
belongs to $B_{p, q}^{\sigma+1}(]0, t_0[, E)$. 
\end{theorem}
See part $\ref{maly}$ for a definition of a local operator
(assumption L2). In the above statement, the microscopic 
q-index plays practically no role, and similar
statements hold within the functional frame of Sobolev spaces
$W^{s, p}$. Nevertheless, this microscopic index
has some importance when dealing 
with critical spaces. Last, for other extensions of
the ODE theory, see for instance
\cite{bohun}, \cite{BJ} and \cite{DL}.

Notice that in the case of operators acting on smooth functions
spaces,
there's no reason to work within
$B_{p, q}^{\sigma}(]0, T[, E)\inj 
B_{\infty, q}^{\sigma-1/p}(]0, T[, E)$, $\sigma-1/p<0$.
Therefore, significant examples of applications of 
theorem $\ref{terminapok}$ must be searched
among irregular operators. Using Bony's decomposition, 
a simple operator is given by
$\mathcal{H}_T(u) = \psi u$ with $\psi\in
B_{p, q}^{\sigma}(]0, T[, \mathscr{L}(E))$,
$1/p<s<\sigma+1<1/p+1$, $s>1/2$ and $s+\sigma>0$. See section
 $\ref{xample}$ for extensions and other examples.
 
Our proof makes use of Picard fixed point theorem.
Since
there's some ``residual'' compactness
 for the above Cauchy problem 
 ($s-\sigma <1 =$ order of derivation), 
 this proof should be routine. Nevertheless, some  technical
difficulties arise, due to
the fractional feature of the problem.
and the low regularity index of the space 
  $B_{p, q}^{\sigma}(]0, T[, E)$. For such a space, multiplication,
  composition and, above all, restriction-extension operations,
  must be handled with care. In particular, all the constants
  of continuity have to be bounded when working on vanishing
  intervals $]0, t[$, $t\rightarrow 0$, even for equivalent
  norms. From this point of view, $B_{p, q}^{-1/p'}(]0, T[, E)$
  seems
  to be a critical space, and most of our proofs
  relies on the following simple fact: the 
  family of zero-extension operators 
$P_{0,t}: B_{p,q}^{\sigma}
(]0, t[, E)\rightarrow B_{p,q}^{\sigma}
(\mathbb{R}, E)$ $(0< t < T)$ is equicontinuous
under the condition 
$1/p-1 <\sigma < 1/p$. 
Equivalently, for such indexes, 
the characteristic function of
an interval is a multiplier for 
$B_{p,q}^{\sigma}
(\mathbb{R}, E)$.

The paper  is organized as follows. In the second section, 
we recall some notations and basic results, 
merely a definition and some 
properties of the Besov spaces,
the definition of the paraproduct
and remainder, and also the definition of some 
duality brackets. The third part is devoted to the statement
of the main theorem. The proof 
of uniform inequalities, essentially a uniform fractional
Poincar\'e's inequality and a fractional integration inequality,
is given in a fourth part.
The firth part contains the proof of the main theorem.
The sixth part concerns some extensions of this theorem.
We state a Peano's type theorem, and also give a global
existence result. Some examples are given in the last part.
\section{\large{Notations and classical 
results.}}\label{nota}

\begin{enumerate}
 \item  Throughout this paper $\mathrm{E}$ and $\mathrm{F}$ 
 are two 
Banach spaces, and $\mathscr{L}(E, F)$ 
 is the space of continuous linear applications from 
 $E$ to $F$ . 
In the sequel, we consider
Banach-valued distributions, and generalize, often
 without comments, scalar results to that context.
The reader is refered to \cite{ama1}, 
\cite{ama2}, but also to \cite{trie1}, \cite{trie2}, 
  \cite{trie3}, 
 \cite{BCD}, since the Banach-valued
case follows from the scalar case by few additional arguments. 
 \item For $1 \leq r \leq \infty$, 
  we denote by $r'$ its conjugate exponent
  i.e $r^{-1}+r^{'-1} = 1$. 
\item The symbol $\hookrightarrow$ stands for classical continuous 
embeddings.
 \item Let $1\leq p, q \leq \infty$ and $s\in \mathbb{R}$. The non-homogeneous Besov space 
$B_{p, q}^{s}(\mathbb{R}^{n}, E)$ can be defined 
as the space of tempered distribution $f$
such that (see\cite{BCD})
$\Vert f 
\Vert_{B_{p, q}^{s}(\mathbb{R}^{n}, E)}
:=
\Vert
(2^{js}\Vert \Delta_{j}f 
\Vert_{L^{p}(\mathbb{R}^{n}, E)})_{j\in \mathbb{Z}}
\Vert_{l^q(\mathbb{Z})}
< \infty$.
In the above writings, the analytic functions 
$\Delta_{j}f $ are defined by the 
standard dyadic procedure (se \cite{BCD} p.99).
In particular, for $j \leq -2$ we have $\Delta_{j}f = 0$.
 For $j\in \mathbb{Z}$, set, for future reference 
$S_{j}f = \sum_{k\leq 
j-1}\Delta_{k}f$. 
\item
For $u
 \in \mathscr{S}'(\mathbb{R}^{n}, \mathscr{L}(E,
F))$, $v
 \in \mathscr{S}'(\mathbb{R}^{n}, E)$, the usual 
 paraproduct (case $E = F = \mathbb{R}$)
 generalizes immediately as:
 $
 \Pi(u, v) = \sum_{j \geq -1}
 S_{j-1}u.\Delta_{j}v
 $
 , and for the remainder:
 $
 \mathfrak{R}(u, v) = 
 \sum_{\vert j-k \vert \leq1}
 \Delta_{j}u.\Delta_{k}v
 $.
 So that formally, we get the Bony decomposition 
 $u.v = \Pi(u, v) + \Pi(v, 
 u) + \mathfrak{R}(u, v)$.
 We shall use freely continuity results
 for the paradaproduct and remainder. See 
 for instance \cite{BCD} pp. 102-104 or \cite{TAY} p.35. 
  \item For $t> 0$, we denote by 
  $\chi_{1/t}\in \mathrm{B}_{m, \infty}^{1/m}
  (\mathbb{R}, \mathbb{R})$ ($1\leq m\leq \infty$) the characteristic
  function of $]0, t[$. We set $\chi = \chi_{1}$. Similarly,
  $\chi_{J}\in\mathrm{B}_{m, \infty}^{1/m}
  (\mathbb{R}, \mathbb{R})$ stands for the characteristic
  function of the interval $J$. Last, 
  $\mathbf{1}_{]0, t[}: ]0, t[\rightarrow\mathbb{R}$
  is the unit function of $]0, t[$.
 \item Let $\Omega \subset \mathbb{R}^{n}$ be a smooth
 domain. The Besov space $B_{p, q}^{s}(\Omega , \mathrm{E})$
 is defined as the restrictions of
  elements of 
 $B^{s}_{p, q}(\mathbb{R}^{n}, \mathrm{E})$ to $\Omega$. The 
 space $B_{p, q}^{s}(\Omega , \mathrm{E})$
 is endowed with the quotient norm
 $\Vert \mathrm{u} 
 \Vert_{B_{p, q}^{s}(\Omega , \mathrm{E})}
 = \mathrm{inf}
\Vert \mathrm{v} 
 \Vert_{B_{p, q}^{s}(\mathbb{R}^{n} , \mathrm{E})} 
 $
 , the inf being taken on all the extensions
 $\uu{v}\in B_{p, q}^{s}(\mathbb{R}^{n} , \mathrm{E})$ of 
 $\uu{u}$. 
 \item Let $\Omega$ be a (smooth)
 domain of $\mathbb{R}^{n}$.
  For any $A \subset
 \mathscr{D}'(\Omega)$, the restriction of  a distribution
 $T\in\mathscr{D}'(\Omega)$ to a domain $\mathscr{O}
  \subset \Omega$ is denoted by 
 $T\vert_{\mathscr{O}}$. The set
 $\uu{A}\vert_{\mathscr{O}}$ is the set 
 of elements $T\vert_{\mathscr{O}}$
 with $T\in\uu{A}$. For $\uu{u}
 \in B_{p, q}^{s}(\Omega , E)$,
 we write $\Vert u 
 \Vert_{B_{p, q}^{s}(\mathscr{O} , E)}
 := 
 \Vert u\vert_{\mathscr{O}}
 \Vert_{B_{p, q}^{s}(\mathscr{O} , E)}$.
\item  We define duality-like pairings. 
The construction is similar 
to the one given in \cite{BCD}, p.70 and p.101 
for the duality bracket. 
We restrict
to the case of an interval $I$,  
 and assume that $1 \leq p, q\leq \infty$,  $-1/p'< \sigma < 1/p$,
 or equivalently $-1/p < -\sigma < 1/p'$. 
It follows that the extension by zero
operators $P_{0, I}$ is continuous in both case: 
\begin{itemize}
 \item $P_{0, I}: B^{-\sigma}_{p', q'}(I,
\mathscr{L}(E, F))
\rightarrow B^{-\sigma}_{p', q'}
(\mathbb{R},
\mathscr{L}(E, F))$
 \item $P_{0, I}: B^{\sigma}_{p, q}(I, E)
\rightarrow B^{\sigma}_{p, q}
(\mathbb{R}, E)$
\end{itemize}
where, as customary,
we have denoted by the same letter
the two operators. Hence, we define the pairing
$<., .>_{\sigma, p, q, I}:B^{-\sigma}_{p', q'}(I,
\mathscr{L}(E, F))
\times B^{\sigma}_{p, q}(I, E)
\rightarrow F$ (or simply, $<., .>_{I}$, 
or $<,>$) by
\begin{align}\label{rouge}
<u, v>_{\sigma, p, q, I} := \sum_{\vert k' - k \vert
\leq 1}\int_{\mathbb{R}}
\Delta_{k}(P_{0, I}(u))(t).\Delta_{k'}(P_{0, I}(v))(t)dt
\end{align}
Hence $<u, v>_{I} = <P_{0, I}(u), P_{0, I}(v)>_{\mathbb{R}}$. 
Function $<., .>_{\sigma, p, q, I}$ "extends"
continuously the pairing
of $L^2(I, \mathscr{L}(E, F))
\times L^2(I, E)
 \rightarrow F$ given by 
$\int_{I} u(t) v(t)
dt$. 

Last, 
we will sometimes write $<v, u>$ in place of $<u, v>$.
\end{enumerate}
\section{\large{Statement of the theorem.}}
\label{maly}
In order to state 
the main 
theorem,  
we have to define  
restriction procedures for an
operator denoted below 
by $\mathcal{H}_T$.
Let
$1\leq p, q \leq \infty$, $0< \alpha
< 1$.
For any $t \in \mathbb{R}_+$, 
$\rho> 0$, $u_0 \in E$,  
define $\mathscr{B}_{t, \alpha}
 (u_0, \rho)$ as the open ball
 of 
 $B_{p, q}^{1/p+\alpha}(]0, t[, E)$
 with center $u_0$ and
 radius $\rho$, and set :
\begin{align}
B_{t, \alpha}
(u_0\mathbf{1}_{]0, t[}, \rho)
= \{
u \in
\mathscr{B}_{t, \alpha}
 (u_0\mathbf{1}_{]0, t[}, \rho)
\textrm{ with }
u(0) = u_0
\}
\end{align}
Denote also by
$\bar{B}_{t, \alpha}
(u_0\mathbf{1}_{]0, t[}, \rho)$ 
its closure in $B_{p, q}^{1/p+ \alpha}(]0, t[, E)$ (similar
notation for 
$\mathscr{B}_{t, \alpha}
 (u_0\mathbf{1}_{]0, t[}, \rho)$). Until the
end of the paper, we often and abusively identify
$u_0\mathbf{1}_{]0, t[}$
with $u_0$. We write for instance, 
$B_{t, \alpha}
(u_0, \rho)$
in place of
$B_{t, \alpha}
(u_0\mathbf{1}_{]0, t[}, \rho)$.
In the sequel, we implicitely use the 
following
\begin{lemma}\label{trict}
Let  
$1\leq p, q \leq \infty$, $0< \alpha < 1$, 
$u_0\in E$, $\rho > 0$ and 
$0\leq t_1\leq t_2$. 
Then, $B_{t_1, \alpha}(u_0, \rho)
=
B_{t_2, \alpha}(u_0, \rho)\vert_{]0, t_1[}$ and 
$\mathscr{B}_{t_1, \alpha}(u_0, \rho)
=
\mathscr{B}_{t_2, \alpha}(u_0, \rho)\vert_{]0, t_1[}$. 
\end{lemma}
\begin{proof} We only prove the first equality. Inclusion 
$
B_{t_2, \alpha}(u_0, \rho)\vert_{]0, t_1[}
\subset
B_{t_1, \alpha}
(u_0, \rho)$ is clear. For the
opposite inclusion, let $u\in 
B_{t_1, \alpha}
(u_0, \rho)$ and let $\epsilon = (\rho -
 \Vert u-u_0
 \Vert_{B_{p, q}^{1/p+\alpha}(]0, t_1[, E)})/2$.
There exists 
$u_* \in B_{p, q}^{1/p+\alpha}(\mathbb{R}, E)$ such that
\begin{align}\label{tct}
u_*\vert_{]0, t_1[} = u
 - u_0
\end{align}
\begin{align}\label{tkt}
 \Vert u_* \Vert_{\mathrm{B}_{\uu{p}, \uu{q}}^{1/p+\alpha}(]0, t_2[, \mathrm{E})}
 \leq 
 \Vert u_* \Vert_{B_{p, q}^{1/p+\alpha}(\mathbb{R}, E)}
 \leq 
 \Vert u-u_0\Vert_{B_{p, q}^{1/p+\alpha}(]0, t_1[, E)} + \epsilon
 < \rho 
\end{align}
Set $u_{**} = u_*+
u_{0, *}$ with 
$u_{0, *}
\in 
B_{p, q}^{1/p+\alpha}(\mathbb{R}, \mathrm{E})
$ and 
$u_{0, *}\vert_{]0, t_2[}
=
u_0$. By $\ref{tct}$, 
$u_{**}(0)
=
u(0) = u_0$.
Hence, by  $\ref{tkt}$,
$u_{**}\vert_{]0, t_2[}
\in
B_{t_2, \alpha}(u_0, \rho)$.
With $\ref{tct}$, this provides
the lemma.
\end{proof}
Let now $R > 0$, $T>0$, 
$1 \leq p, q \leq \infty$,
 $0< \alpha 
< \eta < 1$, and 
$u_0\in E$
be fixed. 
Let 
 $\mathcal{H}_{T}: \mathscr{B}_{T, \alpha}(u_0, R)
 \rightarrow B_{p, q}^{-1/p'+\eta}(]0, T[, E)$. 
Consider the following properties:
for any $(u, v,
t)
 \in 
 \mathscr{B}_{T, \alpha}(u_0, R)^2 
 \times ]0, T[$, we have
\begin{itemize}
 \item ($L1$) $\phantom{OO} \Vert \mathcal{H}_{T}(u) - 
\mathcal{H}_{T}(v)
\Vert_{B_{p, q}^{-1/p'+\eta}(]0, T[, E)}
\leq \mathfrak{C}_T\Vert
u-v
\Vert_{B_{p, q}^{1/p+
\alpha}(]0, T[, E)}$
 \item ($L2$) $\phantom{OO} \textrm{If } u\vert_{]0, t[}
 =
 v\vert_{]0, t[}, \textrm{ then }
 \mathcal{H}_{T}(u)\vert_{]0, t[}
 =
 \mathcal{H}_{T}(v)\vert_{]0, t[}$
\end{itemize}
When condition 
$L2$ is satisfied, 
we define for any $t\in
]0, T[$ an operator:
$$\mathcal{H}_{t}: \mathscr{B}_{t, \alpha}(u_0, R)
 \rightarrow B_{p, q}^{-1/p'+\eta}(]0, t[, E)$$
 by restriction. 
 It means that, for any 
 $u\in\mathscr{B}_{t, \alpha}(u_0, R)$
 we have:
 \begin{align}\label{lox}
  \mathcal{H}_{t}(u)
 =
 \mathcal{H}_{T}(U)\vert_{]0, t[}
 \end{align}
 with $U\in \mathscr{B}_{T, \alpha}(u_0, R)$
 and $U\vert_{]0, \mathrm{t}[} = u$.
 
 With these notations, the main result is the
 following
 \begin{theorem}\label{terminator}
Let  
$T> 0$, $R>0$, 
 $1 \leq p, q \leq \infty$,
  $0 < \alpha< \eta
 < 1$, and $u_0 \in  E$. 
 Assume that  $\mathcal{H}_{T}: 
 \mathscr{B}_{T, \alpha}(u_0, R)
 \rightarrow  B_{p, q}^{-1/p'+
\eta}(]0, T[, E)$ satisfies
 conditions $(L_1)$ and $(L_2)$. 
 Then, there exists $0<\rho
 < R$ and $0 < \mathscr{T} < T$ such that, 
 for any $0<t_0\leq\mathscr{T}$ 
   the 
 problem: find
 $u\in \mathscr{B}_{t_0, \alpha}
 (u_0, \rho)$
 with:
  \begin{equation}\label{seaal}
\begin{cases}
u' = \mathcal{H}_{t_0}
(u)
\\
u(0) = u_0\\
 \end{cases}
\end{equation}
admits exactly one 
distributional solution. This solution
belongs to $B_{p, q}^{1/p+
\eta}(]0, t_0[, E)$. 
\end{theorem}

The proof requires some lemmas which are detailed in the following section.
\section{\large{Uniform estimates.}}
\label{globok}
The main goal of this section is to get uniform (in $t$)
bounds in the required inequalities. 
In the sequel, for $0< t < T$, we denote by
$P_{0,t}: B_{p,q}^{s}
(]0, t[, E)\rightarrow B_{p,q}^{s}
(\mathbb{R}, E)$ the zero-extension operator.
\begin{lemma}\label{clue}
Let $T>0$,
$1 \leq p, q \leq \infty$ and  
$-1/p'< s < 1/p$. The family
$(P_{0,t})_{0< t \leq T}$ is equicontinuous.
\end{lemma}
\begin{proof}
Let $u\in B_{p,q}^{s}
(]0, t[, E)$ and let $\phi\in B_{p,q}^{s}
(\mathbb{R}, E)$ be any extension of $u$. 
Since $-1/p'< s < 1/p$, $\chi_{]0, t[}$ is a multiplier for
$B_{p,q}^{s}
(\mathbb{R}, E)$. Hence, $\phi\chi_{]0, t[}$
is well defined in
$B_{p,q}^{s}
(\mathbb{R}, E)$ and 
$P_{0,t}u = \phi\chi_{]0, t[}$.
Therefore
  \begin{align}
  \Vert
  P_{0,t}u
  \Vert_{B_{p,q}^{s}
(\mathbb{R}, E)}
  \leq 
  C
\big(\Vert
  \chi_{]0, T[}
  \Vert_{B_{p',\infty}^{1/p'}
(\mathbb{R}, E)}+
 \Vert
 \chi_{]0, T[}
  \Vert_{L^{\infty}(\mathbb{R}, E)}
\big)
   \Vert
  \phi
  \Vert_{B_{p,q}^{s}(\mathbb{R}, E)}\label{chaisbeuf}
  \end{align} 
 Taking the inf on all the extensions $\phi\in
  B_{p,q}^{s}(\mathbb{R}, E)$ of $u$, we get
   $$\Vert
  P_{0,t}u
  \Vert_{B_{p,q}^{s}
(\mathbb{R}, E)}
  \leq 
  C_{T}\Vert u
  \Vert_{B_{p,q}^{s}(]0, t[, E)}
  $$
\end{proof}

We deduce from lemma $\ref{clue}$ the integral
formulation of the problem
\begin{lemma}\label{wos}
 Let $T>0$,   
 $1\leq p, q \leq \infty$,
 $1/p < s < 1+1/p$. 
 Let  $\phi\in 
B_{p, q}^{s-1}
 (]0, T[, E)$ and 
 $u_0\in E$. Then, problem: find 
 $u\in B_{p, q}^{s}
 (]0, T[, E)$
 with
 \begin{equation}\label{mapgp}
\begin{cases}
u' = \phi\\
u(0) = u_0\\
 \end{cases}
\end{equation}
admits exactly one solution, given by
\begin{align}\label{butyl}
 u(t) =  
u_0
 + <\phi\vert_{]0, t[}, \mathbf{1}_{]0, t[}>
\end{align}
\end{lemma}
\begin{proof}
We only prove formula 
$\ref{butyl}$.
For $\phi\in C^{\infty}([0, T], E)$, formula 
 $\ref{butyl}$ reduces
 to the usual
 integral formula. In the case
 $\phi\in 
 B_{p, q}^{s-1}
 (]0, T[, E)$, let 
 $(\phi_n)_{n\in\mathbb{N}}$ be a sequence of 
 $\uu{C}^{\infty}([0, T], \mathrm{E})$ functions 
 converging to  $\phi$ in 
 $B_{p, q}^{s-1}
 (]0, T[, E)$. 
 Set  $u(t) =  
u_0
 + <\phi\vert_{]0, t[}, \mathbf{1}_{]0, t[}>$
 and $ u_n(t) =  
u_0
 + <\phi_n\vert_{]0, t[}, \mathbf{1}_{]0, t[}>$. By the continuity of the bracket
$< , >_{\mathbb{R}}$ and the equicontinuity of $(P_{0,t})_{0<t<T}$ 
(lemma $\ref{clue}$), we get, for any $t\in ]0, T[$ (see part $\ref{nota}$,
9.)
  \begin{align}
 \Vert u(t)- u_n(t) \Vert_E
 &\leq C_T
 \Vert \phi-\phi_n \Vert_{B_{p, q}^{s-1}
 (]0, t[, E)}
 \Vert \mathbf{1}_{]0, t[} \Vert_{B_{p', q'}^{1-s}
 (]0, t[, E)}\label{pink}
  \end{align}
  Therefore
    \begin{align}
     \Vert u- u_n \Vert_{L^{\infty}
 (]0, T[, E)}
 \leq C_T
  \Vert \phi-\phi_n \Vert_{B_{p, q}^{s-1}
 (]0, T[, E)}\label{floyd}
    \end{align}
    so that $u_n\rightarrow u$
    in $L^{\infty}
 (]0, T[, E)$.
    Hence, from 
    $u_n' = \phi_n$ we get $u' = \phi$. The rest of the proof 
    is omitted.
\end{proof}

We need two additional fractional inequalities.
The first one (cf. b) in theorem $\ref{chix}$) replaces the full integration in use in the standard
proof of Cauchy-Lipschitz theorem. 
\begin{theorem}\label{chix}
 a) Let $0 < t < T$, $1 \leq m \leq \infty$ 
 and $0\leq  \epsilon < 1/m$.
 Then 
 \begin{align}\label{west}
 \Vert
 \chi_{1/t}
 \Vert_{\mathrm{B}_{m, \infty}^{1/m
 -\epsilon}
 (\mathbb{R})}
 \leq
C_T
 t^{\epsilon} 
 \end{align}

 b) Let $T> 0$, $R> 0$,  
 $1 \leq p, q \leq \infty$,
 $-1/p'< s < \sigma < 1/p$. For any $u \in B_{p,q}^{\sigma}
(]0, T[, E)$
 and any $t
 \in ]0, T]$ we have:
 $ \Vert u \Vert_{B_{p, q}^{s}
 (]0, t[, E)} 
 \leq C_T
 \Vert u \Vert_{B_{p, q}^{\sigma}
 (]0, t[, E)} t^{\sigma-s}
 $.
\end{theorem}

\begin{proof}
 a)  
 \begin{align}
 \Vert
 \chi_{1/t}
 \Vert_{B_{m, \infty}^{1/m
 -\epsilon}}
 =  \Vert
 \chi_{1/T}(t^{-1}T .)
 \Vert_{B_{m, \infty}^{1/m
 -\epsilon}}
 \leq C(T/t)^{1/m-\epsilon-1/m}
 \Vert
 \chi_{1/T}
 \Vert_{B_{m, \infty}^{1/m
 -\epsilon}}\label{wurst}
  \end{align}
  (see \cite{trie1} page 206, or $\ref{hom1}$ below)
  
  b) We
first show that, for 
$0< t < T$,
 and for any $\theta\in
 B_{p, q}^{\sigma}(\mathbb{R}, E)$, the following 
 inequality holds true:
  \begin{align}\label{gonzag}
   \Vert \theta\chi_{1/t}
   \Vert_{B_{p, q}^{s}(\mathbb{R}, E)}
 \leq 
 \mathrm{C}
 \Vert \theta
   \Vert_{B_{p, q}^{\sigma}(\mathbb{R}, E)}
 t^{\sigma - s}
  \end{align}
  
 Set $\epsilon = \sigma - s$. Taking in account $-\epsilon < 0$ and 
 $\sigma-1/p < 0$, we get:
 \begin{align}
B_{p, q}^{\sigma}(\mathbb{R}, E)\times
B_{p, q}^{1/p
 -\epsilon}(\mathbb{R})
 \inj
 B_{\infty, q}^{\sigma-1/p}(\mathbb{R}, E)\times
 B_{p, \infty}^{1/p
 -\epsilon}(\mathbb{R})
  \xrightarrow{\Pi}
  B_{p, q}^{\sigma-\epsilon}
(\mathbb{R}, E)\label{driklet}
  \end{align}
   \begin{align} 
   B_{p, q}^{1/p
 -\epsilon}(\mathbb{R})
   \times
     B_{p, q}^{\sigma}(\mathbb{R}, E)
 \inj
 B_{\infty, \infty}^{
 -\epsilon}(\mathbb{R})
 \times
   B_{p, q}^{\sigma}(\mathbb{R}, E) 
 \xrightarrow{\Pi}
 \mathrm{B}_{p, q}^{\sigma-\epsilon}(\mathbb{R}, E)\label{raklet}
   \end{align}
   Since $\sigma+\frac{1}{p'} - \epsilon = s+1/p'> 0$, 
 we have, for the remainder:
     \begin{align} 
      \mathrm{B}_{p', q}^{1/p'
 -\epsilon}(\mathbb{R})
   \times
    \mathrm{B}_{p, q}^{\sigma}(\mathbb{R}, \mathrm{E})
 \xrightarrow{\mathfrak{R}}
  \mathrm{B}_{1, q}^{\sigma+1/p' - \epsilon }
  (\mathbb{R}, \mathrm{E})
   \inj 
    \mathrm{B}_{p, q}^{\sigma-\epsilon}(\mathbb{R}, E)
    \label{baugar't}
   \end{align}
   Notice that $\chi_{1/t}\in 
 B_{m, q}^{1/m-\epsilon}
 (\mathbb{R})$
 for $m = p$ and for $m = p'$.
 Using $\ref{driklet}$, $\ref{raklet}$, $\ref{baugar't}$ and a),
 inequality $\ref{gonzag}$  a follows.
 
 Now, for $u\in B_{p, q}^{\sigma}(]0, T[, E)$
 and 
 $0< t\leq T$,
 denoting by $\theta \in 
 B_{p, q}^{\sigma}(\mathbb{R}, E)$
  any extension of
$u\vert_{]0, t[}\in B_{p, q}^{\sigma}(]0, t[, E)$ and
 invoking $\ref{gonzag}$:
 \begin{align}\label{Khantt}
  \Vert
 u
  \Vert_{B_{p, q}^{s}(]0, t[, E)}
  \leq
   \Vert
  \theta\chi_{1/t}
  \Vert_{B_{p, q}^{s}(]0, t[, E)}
 \leq
 \mathrm{C}
 \Vert \theta
   \Vert_{B_{p, q}^{\sigma}(\mathbb{R}, E)}t^{\sigma-s}
\end{align}
We take the inf on all the extensions 
$\theta$,
and get b).
\end{proof}
The second inequality is a uniform fractional Poincar\'e's
inequality. We give the proof for a restricted range of values
$1/p<s<1$. The general proof $1/p < s < 1+1/p$ relies 
on tedious extension-retractation arguments and is omitted.

In the proof, for any open subset
$\Omega$ of $\mathbb{R}$, we use the function $I_{\lambda, \Omega}: \Omega
 \rightarrow \mathbb{R}$ (or simply 
 $I_{\lambda}$) 
 defined by $I_{\lambda}(t)
 = \lambda t$. Recall
 the following 
inequality (see \cite{trie1}),
valid for any  
$\lambda \geq 1$, $1\leq p, q \leq \infty$,  
$s>0$ and $u\in B_{p, q}^{s}
(\Omega, E)$
\begin{align}
\Vert
uoI_{\lambda}
\Vert_{B_{p, q}^{s}(I_{1/\lambda}(\Omega), E)}
\leq C \lambda^{s-1/p}
\Vert
u
\Vert_{B_{p, q}^{s}(\Omega, E)}
\label{hom1}
\end{align}
By a duality argument, this 
inequality holds true for
$s<0$ and $0< \lambda \leq 1$ (see \cite{BCD} prop. 2.76). 
Therefore
\begin{lemma}\label{derche}
Assume that  $1/p < s <1+1/p$, and let 
$\Omega$ be a bounded interval of $\mathbb{R}$.
There exists $C_{\Omega}>0$ such that for any
$0<\lambda\leq 1$ and any 
$u\in B_{p, q}^{s}
(I_{\lambda}(\Omega), E)$ with $u(0) = 0$, we have
\begin{align}
\Vert
u
\Vert_{B_{p, q}^{s}
(I_{\lambda}(\Omega), E)}
\leq C_{\Omega} \Vert
u'
\Vert_{B_{p, q}^{s-1}
(I_{\lambda}(\Omega), E)}
\label{homma}
\end{align}
\end{lemma}
\begin{proof}
a) Assume that  $1/p < s <1+1/p$. We first prove inequality $\ref{homma}$ 
for $\lambda = 1$. 
Theorem 3.3.5, p. 202 
in \cite{trie1} gives
\begin{align}
\Vert u 
\Vert_{B_{p, q}^{s}
(\Omega, E)}
&\leq
\uu{C}_{\Omega}
\big(
\Vert u'
\Vert_{B_{p, q}^{s-1}
(\Omega, E)}
+ 
\Vert u
\Vert_{B_{p, q}^{s-1}
(\Omega, E)}
\big)\label{grd}
\end{align}
Note that $$
\Vert u
\Vert_{B_{p, q}^{s-1}
(\Omega, E)}\leq 
\epsilon \Vert u
\Vert_{B_{p, q}^{s}
(\Omega, E)}+
C_{\epsilon}
\Vert u
\Vert_{L^p
(\Omega, E)}$$
for $\epsilon>0$ arbitrary small. With $\ref{grd}$, it provides
\begin{align}
\Vert u 
\Vert_{B_{p, q}^{s}
(\Omega, E)}
&\leq
\uu{C}_{\Omega}
\big(
\Vert u'
\Vert_{B_{p, q}^{s-1}
(\Omega, E)}
+ 
\Vert u
\Vert_{L^p
(\Omega, E)}
\big)\label{bartav}
\end{align}
Next, arguing as in 
$\ref{pink}$ and $\ref{floyd}$, we get
\begin{align}
 \Vert u
\Vert_{L^p
(\Omega, E)}
\leq 
C_T 
\Vert u'
\Vert_{B_{p, q}^{s-1}
(\Omega, E)}\label{bartaba}
\end{align}
Therefore, the case $\lambda=1$, follows
from $\ref{bartav}$ and $\ref{bartaba}$.

b) Assume that $1/p<s<1$.
In the general case $0< \lambda \leq 1$, set
$v = uoI_{\lambda}$. We have
\begin{align}
\Vert
u
\Vert_{B_{p, q}^{s}
(I_{\lambda}(\Omega), E)}
& =
\Vert
voI_{1/\lambda}
\Vert_{B_{p, q}^{s}
(I_{\lambda}(\Omega), E)}\nonumber\\
&\leq
C(1/\lambda)^{s-1/p}
\Vert
v
\Vert_{B_{p, q}^{s}
(\Omega, E)}\label{interme}
\end{align}
due to $\ref{hom1}$ since 
$s>0$ and $\lambda^{-1}\geq 1$.
Next, the case 
$\lambda = 1$ provides
\begin{align}
\Vert
v
\Vert_{B_{p, q}^{s}
(\Omega, E)}
&\leq
C_{\Omega}
\Vert
v'
\Vert_{B_{p, q}^{s-1}
(\Omega, E)}\nonumber\\
&= C_{\Omega}\lambda
\Vert
u'oI_{\lambda}
\Vert_{B_{p, q}^{s-1}
(\Omega, E)}
\nonumber\\
& \leq C C_{\Omega}\lambda^{s-1-1/p}
\lambda 
\Vert
u'
\Vert_{B_{p, q}^{s-1}(
I_{\lambda}(\Omega), E)}\label{presq}
\end{align}
by  
inequality $\ref{hom1}$ since $0 < \lambda \leq 1$ and
$s-1 < 0$. 
Inequality $\ref{homma}$ follows
from
$\ref{interme}$ and $\ref{presq}$.
\end{proof}

\section{\large{Proof of the theorem.}}
\label{izno}
Before proceeding, we need a last uniform lemma.

 \begin{lemma}\label{bfmwc}
 Let  $T>0$,   
 $1 \leq p, q < \infty$,
  $0< \alpha 
  < \eta < 1$ , 
 $u_0 \in E$ and 
$R > 0$.
 Let also 
 $\mathcal{H}_{T}: 
 B_{T, \alpha}
 (u_0, R)
 \rightarrow B_{p, q}^{-1/p'+
\eta}(]0, T[, E)$ satisfies properties
 $L1$ and $L2$. Then, there exists
 $\Lambda_T >0$ such that for any $0< t <T$
 and 
 $(u, v)
 \in 
 \mathrm{B}_{t, \alpha}(u_0, \mathrm{R}/
 \Lambda_T)^2$
we have
 $$\Vert \mathcal{H}_{t}(u) - 
 \mathcal{H}_t(v)
\Vert_{B_{p, q}^{-1/p'+
\eta}(]0, t[, E)}
\leq \mathfrak{C}_T\Lambda_T
\Vert
u-v
\Vert_{B_{p, q}^{1/p+
\alpha}(]0, t[, E)}$$
 \end{lemma}
\begin{proof}
a) We first define an equicontinuous family 
$(Q_t)_{0< t< T}$ of extension operators.
Let $\theta \in 
\mathscr{D}(\mathbb{R}, \mathbb{R}_+)$
with $\theta (s) = 1$ for
$\vert s \vert \leq 2T$.
For any $0< t <T$ and $u\in B_{p, q}^{1/p+
\alpha}(]0, t[, E)$  set
$$
Q_t(u)(\tau) = 
\theta(\tau)
(<P_{0, t}(u'), \chi_{]0, \tau[}>+u(0)
)$$
$(\tau\in\mathbb{R})$. It follows from lemma $\ref{wos}$ that
$
Q_t: B_{p, q}^{1/p+
\alpha}(]0, t[, E)\rightarrow
B_{p, q}^{1/p+
\alpha}(\mathbb{R}, E)
$ is an extension operator. Moreover, 
due to the continuity of the bracket
$<,>_{\mathbb{R}}$ and the equicontinuity
of $(P_{0, t})_{0<t<T}$ (see lemma
$\ref{clue}$), we have
$$
\Vert
Q_t(u)
\Vert_{B_{p, q}^{1/p+
\alpha}(\mathbb{R}, E)}
\leq
C_T  \big(
\Vert
u'
\Vert_{B_{p, q}^{-1/p'+
\alpha}(]0, t[, E)}+\vert u(0) \vert\big)
\leq
C_T 
\Vert
u
\Vert_{B_{p, q}^{1/p+
\alpha}(]0, t[, E)}
$$
(the last inequality holds with
$\mathbb{R}$ in place of $]0, t[$, 
and follows on $]0, t[$ using by the definition of the norms). 
Hence, $(Q_t)_{0< t< T}$ is equicontinuous. We denote by $\Lambda_T$ a bound
of the norms of the $Q_t$.

b) Let now $u\in \mathrm{B}_{t, \alpha}(u_0, \mathrm{R}/
 \Lambda_T)$. Note that $u_0\mathbf{1}_{]0, T[}
 = Q_t(u_0\mathbf{1}_{]0, t[})\vert_{]0, T[}$.
 With a), it provides $\Vert Q_t(u)- u_0
 \Vert_{B_{p, q}^{\frac{1}{p}+
\alpha}(]0, T[, E)}
\leq \Lambda_T \Vert u - u_0
 \Vert_{B_{p, q}^{\frac{1}{p}+
\alpha}(]0, t[, E)}$,
hence
$$Q_t\big(\mathrm{B}_{t, \alpha}(u_0, R/
 \Lambda_T)\big\vert_{]0, T[}
 \subset \mathrm{B}_{T, \alpha}(u_0, R)$$
Taking another $v\in \mathrm{B}_{t, \alpha}(u_0, R/
 \Lambda_T)$, we get
\begin{align}
 \Vert \mathcal{H}_{t}(u) - 
 \mathcal{H}_t(v)
\Vert_{B_{p, q}^{-1/p'+
\eta}(]0, t[, E)}
&\leq 
\Vert
\mathcal{H}_{T}\big(Q_t(u)\vert_{]0, T[}\big) - 
 \mathcal{H}_T\big(Q_t(v)\vert_{]0, T[} \big)
\Vert_{B_{p, q}^{-1/p'+
\eta}(]0, T[, E)}\nonumber\\
&\leq 
\mathfrak{C}_T
\Vert
Q_t(u) - 
Q_t(v)
\Vert_{B_{p, q}^{1/p+
\alpha}(]0, T[, E)}\nonumber\\
&\leq
\mathfrak{C}_T\Lambda_T
\Vert
u - 
v
\Vert_{B_{p, q}^{1/p+
\alpha}(]0, t[, E)}\nonumber
\end{align}
\end{proof}
We now prove theorem 
$\ref{terminator}$. 
\begin{proof}\label{fine}
We use Picard fixed
point theorem. 
Let  $0<\rho < R/\Lambda_T$
and let $0< t_0 < T$ to be precised.
Define:
\begin{equation}\label{ghap}
\mathcal{S}_{t_0}:
\begin{cases}
\bar{B}_{t_0,\alpha}
(u_{0}, \rho)
&\rightarrow 
B_{p, q}^{1/p+\alpha}
(]0, t_0[, E)\\
\phantom{VVVVV}\tilde{u} &\mapsto 
u\\
 \end{cases}
\end{equation}
where $u$ is given by equation $\ref{butyl}$
with $\mathcal{H}_{t_0}(\tilde{u})
$
in place of $\phi$ and $s = 1/p+\eta>
1/p+\alpha$.\\
 We prove that 
 $\mathcal{S}_{t_0}(\bar{B}_{t_0,\alpha}
(u_{0}, \rho))\subset\bar{B}_{t_0,\alpha}
(u_{0}, \rho)$ for $t_0>0$ small enough.
 
 Let $\tilde{u}\in
 \bar{B}_{t_0,\alpha}(u_0, \rho)$ and 
 $u = \mathcal{S}_{t_0}(\tilde{u})$.
Appealing to lemmas $\ref{wos}$, $\ref{derche}$
and theorem
$\ref{chix}$ 
for $0< \alpha <\eta$
we have:
 \begin{align}
  & \Vert u - u_0
   \Vert_{B_{p, q}^{1/p+\alpha}
(]0, t_0[)}
   \leq
   C_T
   \Vert u'
  \Vert_{B_{p, q}^{-1/p'+\alpha}
(]0, t_0[)}
   =
    C_T
     \Vert \mathcal{H}_{t_0}
     (\tilde{u})
   \Vert_{B_{p, q}^{-1/p'+\alpha}
(]0, t_0[)}
   \label{Rag}\\
   &\leq 
     C_T
{t_0}^{\eta-\alpha}
     \big(\Vert \mathcal{H}_{t_0}
   (\tilde{u})
   - \mathcal{H}_{t_0}
   (u_0)
   \Vert_{B_{p, q}^{-1/p'+\eta}
(]0, t_0[)}
   +
   \Vert \mathcal{H}_{t_0}
   (u_0)
   \Vert_{B_{p, q}^{-1/p'+\eta}
(]0, t_0[)}\big)\label{commun}
  \end{align}
Due to lemma $\ref{bfmwc}$ and
$ \Vert \tilde{u}
  -u_0
   \Vert_{B_{p, q}^{1/p+\alpha}
(]0, t_0[)} \leq \rho$, 
   we get:
 \begin{align}
   \Vert u - u_0
   \Vert_{B_{p, q}^{1/p+\alpha}
(]0, t_0[)}
   &\leq  
C_T
 {t_0}^{\mathrm{\eta-\alpha}}
    ( 
    \rho
   +
   \Vert \mathcal{H}_{T}
   (u_0)
   \Vert_{B_{p, q}^{-1/p'+\eta}
(]0, T[)})
   \leq \rho\nonumber
   \end{align}
   for $t_0>0$ small enough. It proves the stability.
  The proof that  $\mathcal{S}_{t_0}$
   is a contraction is similar.
\end{proof}
\section{\large{Generalization.}}
\label{tenssion}
 Theorem $\ref{terminator}$ is not satisfactory for an operator
  $\mathcal{H}_{T}: U
 \subset
 B_{p, q}^{1/p+\alpha}
(]0, T[, E)
\rightarrow  B_{p, q}^{-1/p'+\eta}
(]0, T[, E)$ defined
 on an arbitrary open  set $U$.
We have to identify
 $$I(U)
 :=
 \{
 u_0\in E 
 \textrm{ such that 
 there exists }
 0< T_{u_0}\leq 
 T \textrm{ with }
 u_0\mathbf{1}_{]0, T_{u_0}[}
 \in 
 U\vert_{]0, 
 T_{u_0}[}
 \}
 $$
The following proposition asserts that
$I(U)= U(0)$ (set of initial values
of elements of $U$) and provides
a uniform estimates on the time $T_{u_0}$.
\begin{proposition}\label{trasse}
Let $T> 0$, 
$1\leq p, q \leq \infty$, 
$0< \alpha < 1$, and let
$U$ be an open subset of
$B_{p, q}^{1/p+\alpha}
(]0, T[, E)$. Then
$I(U)= U(0)$. 
Moreover, for any $u_0
\in I(U)$ there exists
$\gamma>0$, $R>0$ and $T_0>0$
such that for any $u_1
\in E$ with $\Vert
u_1-u_0\Vert_{E}
\leq \gamma$, we have 
\begin{align}\label{janvier}
u_1\mathbf{1}_{]0, T_0[}
\in \mathscr{B}_{T_0, \alpha}
(u_0\mathbf{1}_{]0, T_0[},
R)\subset U\vert_{]0, T_0[}
\end{align}
\end{proposition}
\begin{proof}
The inclusion 
$I(U)\subset U(0)$
is clear.
We prove
the reverse inclusion -i.e that for any 
 $u\mathbf{1}_{]0, T_0[}\in U$, $u(0)\in 
U\vert_{]0, T_0[}$ for some
 $0<T_0\leq T$- and 
 $\ref{janvier}$ at the same time. 

Let $u\in U$. For $R> 0$
small enough, we have $\mathscr{B}
_{\uu{T}, \alpha}(u, R)
\subset U$. Denote by $C_{\infty}$
a constant of continuity
for the embeddings
$B_{p, q}^{1/p+\alpha}
(]0, T[, E)
\inj 
L^{\infty}(]0, T[, E)$ and 
$E\inj B_{p, q}^{1/p+\alpha}
(]0, T[, E)$, and set
$\epsilon = R/ \big[
4(2C_{\infty}+1)\big]$.
Pick up $\psi_{\epsilon}
\in \mathscr{B}
_{T, \alpha}(u, R)
\cap C^{\infty}([0, T], E)$
with
$\Vert
u-\psi_{\epsilon}
\Vert_{\mathrm{B}_{p, q}^{\frac{1}{p}+\alpha}
(]0, T[, E)}
\leq 
\epsilon$
and define $\phi_{\epsilon} := \psi_{\epsilon}
-\psi_{\epsilon}(0)+u(0)$. We have
$\Vert
u(0)-\psi_{\epsilon}(0)
\Vert_{E}
\leq 
\uu{C}_{\infty}
\Vert u-\psi_{\epsilon}
\Vert_{B_{p, q}^{1/p+\alpha}
(]0, T[, E)}
$. By definition of
$\phi_{\epsilon}$ and 
$\psi_{\epsilon}$, this implies that 
$\Vert u-\phi_{\epsilon}
\Vert_{B_{p, q}^{1/p+\alpha}
(]0, T[, E)}
\leq(C_{\infty}+1)\epsilon$.

Let now $u_1 \in E$
with $\Vert u_1-u(0) \Vert_{E}
\leq \epsilon
$
, and let
$\alpha< \delta <1$. Since
$\phi_{\epsilon}(0) = 
u(0)$, appealing
to theorem $\ref{chix}$ b), we get,
for any $0< t< T$
\begin{align}
\Vert
u-u_1
\Vert_{B_{p, q}^{1/p+\alpha}
(]0, t[)}
&\leq 
\Vert
u-\phi_{\epsilon}
\Vert_{B_{p, q}^{1/p+\alpha}
(]0, t[)}
+
\Vert
\phi_{\epsilon}-\phi_{\epsilon}(0)
\Vert_{B_{p, q}^{1/p+\alpha}
(]0, t[)}+
\Vert
u(0)-
u_1\Vert_{B_{p, q}^{1/p+\alpha}
(]0, t[)}\nonumber\\
&\leq
(C_{\infty}+1)\epsilon
+
C_{T}
\Vert
\phi_{\epsilon}'
\Vert_{B_{p, q}^{-1/p'+\delta}
(]0, t[)}
t^{\delta-\alpha} 
+C_{\infty}\epsilon\label{sonate}
\end{align}
Set $
T_{0}
=inf
\Big\{
\Big(C_T^{-1}
\Vert
\phi'_{\epsilon}
\Vert_{B_{p, 
q}^{-1/p'+\delta}
(]0, T[)}^{-1}[R/2-(2C_{\infty}+1)\epsilon]\big)
^{1/(\delta - \alpha)}
, T\Big\}$. 
From inequality
$\ref{sonate}$ and lemma $\ref{trict}$
2) we get $u_1
\in \mathscr{B}
_{T_{0}, \alpha}(u, R)
 = \mathscr{B}
_{T, \alpha}(u, R)
\vert_{]0, T_{0}[}
\subset U\vert_{]0, T_{0}[}$,
which proves the proposition.
\end{proof}
Due to corollary $\ref{terminator}$ , lemma
$\ref{trasse}$ (and 
uniform estimates of the time of existence in the above proofs), we get corollary 
$\ref{terminable}$ below. We extend
without comments the range of indexes,
since the proof is easier for spaces
$B_{p, q}^{s}$ of positive 
differential dimension $s-1/p\geq 0$.
We also give a statement in the case of 
a continuous operator $\mathcal{H}_{T}$.
\begin{corollary}\label{terminable}
Let  
$T> 0$, $R>0$, 
 $1 \leq p, q \leq \infty$,
  $\eta >\alpha> 0$, and let $U$ be an
 open subset of 
 $B_{p, q}^{1/p+\alpha}
(]0, T[, E))$. Let $u_0 \in U(0)$.\\

a)
 Assume that  $\mathcal{H}_{T}: 
 U
 \rightarrow  B_{p, q}^{-1/p'+
\eta}(]0, T[, E)$ satisfies
 condition $L_1$ and $L_2$. 
 Then, there exists $0<\rho_1 \leq \rho_2
 < R$ and $0 < \mathscr{T} < T$ such that, 
 for any $0<t_0\leq\mathscr{T}$ and any
  $u_1 \in E$ with 
  $\Vert u_1 - u_0 \Vert \leq \rho_1$, 
   the 
 problem: find
 $u\in \mathscr{B}_{t_0, \alpha}
 (u_0, \rho_2)$
 with:
  \begin{equation}\label{sebal}
\begin{cases}
u' = \mathcal{H}_{\mathrm{t}_0}
(u)
\\
u(0) = u_1\\
 \end{cases}
\end{equation}
admits exactly one 
solution. This solution
belongs to $B_{p, q}^{1/p+
\eta}(]0, t_0[, E)$.\\

b) Assume that 
$U
 = B_{p, q}^{1/p+
\alpha}(]0, T[, E)$
in a). Then, the solution exists
on the whole interval $]0, T[$.

c) Same assumptions as in a) except that $E$ is finite dimensional and
$\mathcal{H}$ is not Lipschitzian but
continuous. In the
conclusions of a), uniqueness is lost.
\end{corollary}
\begin{proof}
 We only prove b). Appealing to standard arguments, 
 it's enough to get a priori bounds in
$B_{p, q}^{1/p+
\eta}$ for a local solution $u$ defined
 on an interval $[0, t_0[$. For $0< t< t_0$, 
 arguing as in the proof of theorem 
$\ref{terminator}$, we get, 
   for any $0<t<t_0$
 \begin{align}
  \Vert u - u_0
   \Vert_{B_{p, q}^{1/p+\eta}
(]0, t[)} 
&\leq 
     C_T \Vert u ' 
     \Vert_{B_{p, q}^{-1/p'+\eta}(]0, t[)}\nonumber\\
   &\leq 
     C_T
     \big(\Vert \mathcal{H}_{t}
   (u)
   - \mathcal{H}_{t}
   (u_0)
   \Vert_{B_{p, q}^{-1/p'+\eta}
(]0, t[)}
   +
   \Vert \mathcal{H}_{t}
   (u_0)
   \Vert_{B_{p, q}^{-1/p'+\eta}
(]0, t[)}\big)\nonumber\\
&\leq C_T\big(\Vert u - u_0
   \Vert_{B_{p, q}^{1/p+\alpha}
(]0, t[)} +1\big)
\nonumber\\
&\leq C_T\big(\epsilon \Vert u - u_0
   \Vert_{B_{p, q}^{1/p+\eta}
(]0, t[)} +A_{\epsilon}\Vert u - u_0
   \Vert_{L^p(]0, t[)}+1\big)
\label{arnaurd}
  \end{align}
  We do not prove that the constant $A_{\epsilon}$
  is independent of $t$. This can be done by 
  using the operators $Q_t$ (see proof of lemma 
  $\ref{bfmwc}$). 
  Hiding the term $C_T\epsilon \Vert u - u_0
   \Vert_{B_{p, q}^{1/p+\eta}(]0, t[)}$ in the left hand side
   of $\ref{arnaurd}$, and
   using  $B_{p, q}^{1/p+\eta}(]0, t[)
   \inj L^{\infty}(]0, t[)$ 
   (equicontinuous family of embeddings) we get
    \begin{align}
     \Vert
     u(t)
     \Vert_E
     \leq C \Vert u
   \Vert_{B_{p, q}^{1/p+\eta}
(]0, t[)}
     \leq C_T\big(\Vert
     u
     \Vert_{L^p(]0, t[)}+1\big)\label{tonyo}
    \end{align}
Set $z(t) = \Vert u
     \Vert_{L^p(]0, t[)}$.
     Inequality $\ref{tonyo}$ implies $z'(t)
     \leq C_T(z(t)+1)$. By Gronwall lemma, it follows
     that $z$ is bounded on
     $[0, t_0[$. Invoking one again 
     $\ref{tonyo}$, we get that 
     $\Vert u
   \Vert_{B_{p, q}^{1/p+\eta}
(]0, t[)}$ is bounded on $[0, t_0[$
 
\end{proof}

\section{\large{Examples.}}
\label{xample}
Examples of operators $\mathcal{H}_{T}: 
B_{p, q}^{1/p+\alpha}
(]0, T[, E)
 \rightarrow  B_{p, q}^{-\frac{1}{p'}+
\eta}(]0, T[, E)$ endowed with properties
$L1$ and $L2$ can be searched by means of 
Fourier series
$$
\mathcal{H}_{T}(u) = 
\sum_{j\in \mathbb{Z}}
c_j(u) e_j
$$
 with $e_j(t) = exp(2ij\pi t/T)$
 and where the $c_j: B_{p, q}^{1/p+\alpha}
(]0, T[, E)\rightarrow \mathbb{C}$ are continuous
functions.
Nevertheless, property $L2$ is 
not easily characterized, and
other decompositions must be looked for.
See examples c) and d) below. We will need the following

\begin{lemma}\label{para}
Assume that 
$1 \leq p, q \leq \infty$, $-1/p'< \sigma \leq 1/p$,   
$1/p < s \leq 1/p+1$ with $\sigma + s> 0$
Then, for any open interval $I$,
and any 
$\psi\in B_{p, q}^{\sigma}(I, 
\mathscr{L}(F, E))$,
$u \in
B_{p, q}^{s}(I, F)$, we have
\beg\label{crisp}
\Vert \psi u
\Vert_{B_{p, q}^{\sigma}(I, 
E)}
\leq C
\Vert \psi \Vert_{B_{p, q}^
{\sigma}(I,
\mathscr{L}(F, E))}
\Vert u  \Vert_{B_{p, q}^{s}(I,
F)} 
\end{align}
\end{lemma}
\begin{proof}
We only treat the case $I=\mathbb{R}$. 
For the remainder term, assume first that 
$p \geq 2$. Then
\begin{align}
 B_{p, q}^{\sigma} \times 
B_{p, q}^{s} 
\xrightarrow{\mathfrak{R}}
B_{p/2, \infty}^{\sigma+s}
\inj
B_{p, q}^{\sigma+s-1/p}
\inj 
B_{p, q}^{\sigma}\nonumber
\end{align}
Since $s+\sigma > 0$. And for $1 \leq p\leq 2$
\begin{align}
B_{p, q}^{\sigma} \times 
B_{p, q}^{s} 
\inj
B_{p, q}^{\sigma}\times 
B_{p', \infty}^{s - 1/p+1/p'}
&\xrightarrow{\mathfrak{R}}
B_{1, q}^{s-1/p+1/p'+\sigma}
\inj
B_{p, q}^{s-1/p+\sigma}
\inj
B_{p, q}^{\sigma}\nonumber
\end{align}
since $(s-1/p)+(1/p'+\sigma) >0$. And for the paraproducts  
(see \cite{BCD},
p.103):
$$
B_{p, q}^{s}
\times
B_{p, q}^{\sigma}
\inj
\mathrm{L}^{\infty}\times B_{p, q}^{\sigma}
\xrightarrow{\Pi}
B_{p, q}^{\sigma}
$$
$$
B_{p, q}^{\sigma} \times 
B_{p, q}^{s}
\inj
B_{\infty, q}^{\sigma -1/p}
\times B_{p, \infty}^{s}
\xrightarrow{\Pi}
B_{p, q}^{s-1/p+\sigma }
\inj
B_{p, q}^{\sigma}$$
since $\sigma-1/p<0$ and $s-1/p>0$.
\end{proof}

In particular, the product
is well defined and continuous in 
\begin{align}
  B_{\infty, \infty}^{s}(]0, T[, \mathbb{R})
  \times  B_{\infty, \infty}^{\sigma}(]0, T[, \mathbb{R})
  \rightarrow
   B_{\infty, \infty}^{\sigma}(]0, T[, 
   \mathbb{R})\label{pline}
\end{align}
for $s>1/2$ and $\sigma>-1/2$.

We now proceed with the examples. In the sequel, we restrict our
use of theorem $\ref{terminable}$ to
the initial range of values
$1\leq p, q \leq \infty$ and 
$0<\alpha < \beta < 1$.

a) (Cauchy-Lipschitz, see \cite{arn}, \cite{DD}) 
Let $\Omega$ be an open subset of
$E$ and $f: \Omega \rightarrow E$
be a Lipschitz function. Define an operator
$$\mathcal{H}_T: 
 U\subset
  B_{p, q}^{1/p
 +
\alpha}(]0, T[, E)
 \rightarrow  B_{p, q}^{1/p
 +
\alpha}(]0, T[, E)$$
by $\mathcal{H}_T(u) = f(u)$.
Here, $U:=\{v\in
B_{p, q}^{1/p
 +
\alpha}(]0, T[, E)
\textrm{ s.t }
v([0, T])\subset\Omega
\}$. 
Operator $\mathcal{H}_T$ satisfies properties
$L1$ and $L2$. Hence, problem
$\ref{seaal}$ is locally
well posed
in  
$B_{p, q}^{1/p
 +
\alpha}$. This solution belongs to   
$W^{2, \infty}$ and is unique in this 
space.\\

b) Recall that for 
$0<\beta < 1$, 
$B_{\infty, \infty}^{\beta}(]0, T[)$
is the space of bounded Holder functions with exponent $\beta$. 

In this b), we consider an
operator of the form 
$\mathcal{H}_{T}(u)=A(u)D^{\beta_{}}u$, where
$D^{\beta}$ is either the Riemann-Liouville
either the Caputo fractional derivative,
which we now define.

Let $T>0$, $0<\beta<1$, and let $u\in L^{\infty}(]0, T[,
\mathbb{R})$.  We set $$(J^{1-\beta}u)(t) = 
\int_0^t (t-s)^{-\beta}u(s)ds$$
$0< t \leq T$.
For $u\in W^{s, \infty}(]0, T[,
\mathbb{R})$ ($1/2< s\leq 1$),
define the Caputo derivative of order 
$\beta$ as 
\begin{align}
D^{\beta}_{c}u =
\big(J^{1-\beta} (u-u(0))\big)'/\Gamma(1-\beta)\label{caputo}
\end{align}
Let $0< t < t+h < T$. For 
$u\in L^{\infty}(]0, T[, \mathbb{R})$, we have 
\begin{align}
&\vert
J^{1-\beta}u(t+h)-
J^{1-\beta}u(t)
\vert\nonumber\\
&\leq
\Vert
u
\Vert_{L^{\infty}}
\int_0^t
(
(t-s)^{-\beta}
-
(t+h-s)^{-\beta}
)ds
+
\Vert
u
\Vert_{L^{\infty}}
\int_t^{t+h}
(t+h-s)^{-\beta}
ds\nonumber\\
&\leq
C_T
\Vert
u
\Vert_{L^{\infty}}\vert h \vert^{1-\beta}\label{prosit}
\end{align}
and similarly for $h<0$. It follows that
\begin{align}
J^{1-\beta}: L^{\infty}(]0, T[)
\rightarrow
B_{\infty, \infty}^{1-\beta}(]0, T[, \mathbb{R})\label{kontinue}
\end{align}
continuously. 
Notice now that, for any
$u\in W^{1, \infty}(]0, T[,
\mathbb{R})$, we have
$D^{\beta}_{c}u(t) = 
\int_0^t u'(s)/(t-s)^{\beta}ds$. Hence, 
$\ref{kontinue}$ provides
$$
D^{\beta}_{c}:
W^{1, \infty}(]0, T[, \mathbb{R})
\rightarrow
B_{\infty, \infty}^{1-\beta}(]0, T[, \mathbb{R})
$$
continuously. Similarly, for
$u\in W^{s, \infty}(]0, T[, \mathbb{R})$,
$1/2< s< 1$, 
using formulas
$\ref{caputo}$ and $\ref{kontinue}$, we obtain
 $$
D^{\beta}_{c}:
W^{s, \infty}(]0, T[, \mathbb{R})
\rightarrow
B_{\infty, \infty}^{-\beta}(]0, T[, \mathbb{R})
$$
continuously. By standard embeddings and real interpolation
formulas (see 
\cite{trie1} p. 204) we get, for $\epsilon > 0$
small enough
 $$
D^{\beta}_{c}:
B^{1-\epsilon}_{\infty, \infty}(]0, T[, \mathbb{R})
\rightarrow
B_{\infty, \infty}^{-\epsilon/2}(]0, T[, \mathbb{R})
$$
continuously. 

We now work with vector valued fuctions, for which
the above notations and results can readily
be extended. Until the end of this b), $n\in \mathbb{N}^*$ 
and $A\in W^{1, \infty}\big(\mathbb{R}^n,  
\mathbf{M}_{n\times n}(\mathbb{R})\big)$ are fixed.
For 
$0< \beta_j< 1$ $(1\leq j \leq n)$ and 
$u\in W^{s, \infty}(]0, T[, \mathbb{R}^n)$, $1/2< s< 1$,
we write
$D^{\beta}_{c}u  = (D^{\beta_j}_{c}u_j)_{1\leq j\leq n}$.
In the sequel, $\beta^* = sup(\beta_1, ..., \beta_n)$.
 Using $\ref{pline}$,
we can define, for $\epsilon > 0$
small enough
$$
\mathcal{H}_{T}:
B^{1-\epsilon}_{\infty, \infty}(]0, T[, \mathbb{R}^n)
\rightarrow
B_{\infty, \infty}^{-\epsilon/2}(]0, T[, \mathbb{R}^n)
$$
by
$
\mathcal{H}_{T}(u) = A(u)D^{\beta}_{c}u$.
Operator $
\mathcal{H}_{T}$ satifies properties 
$L1$ and $L2$. Hence, under condition 
$0< \beta_j< 1$, problem $\ref{seaal}$ admits a unique solution 
in $B^{1-\epsilon}_{\infty, \infty}(]0, t_0[, \mathbb{R}^n)$.
This solution belongs to 
$B^{2-\beta^* }_{\infty, \infty}(]0, t_0[, \mathbb{R}^n)$.

Similarly, the Riemann-Liouville derivative is given, 
for $u\in L^{\infty}(]0, T[, \mathbb{R})$, by 
\begin{align}
D^{\beta}_{r}u =
\big(J^{1-\beta}u\big)'/\Gamma(1-\beta)\label{rl}
\end{align}
Using $\ref{kontinue}$ we get, for $1/2<s\leq 1$
\begin{align}
D^{\beta}_{r}:
B^{s}_{\infty, \infty}(]0, T[, \mathbb{R})
\rightarrow
B_{\infty, \infty}^{-\beta}(]0, T[, \mathbb{R})\label{loupe}
\end{align}
continuously. Coming back to the vector-valued case, 
let $0< \beta_j <1/2$  and let
$\epsilon>0$ such that
$1/2<1-\beta_j-\epsilon<1$ $(1\leq j \leq n)$. Using 
$\ref{loupe}$ and $\ref{pline}$, 
we get that 
$$
\mathcal{H}_T:
B^{1-\beta^* -\epsilon}_{\infty, \infty}(]0, T[, \mathbb{R}^n)
\rightarrow
B_{\infty, \infty}^{-\beta^*}(]0, T[, \mathbb{R}^n)
$$
with $\mathcal{H}_T(u) = A(u) D_{r}^{\beta}u$ is well defined.
Operator $
\mathcal{H}_{T}$ satifies properties 
$L1$ and $L2$.  Therefore, under condition 
$0< \beta_j< 1/2$, problem 
$\ref{seaal}$ is locally
well posed, with a solution $u\in B^{1-\beta^*}_{\infty, \infty}
(]0, t_0[)$. 

c) We assume here that $1 < p < \infty$. In this c), 
we abusively write 
$\chi_{]s, T[}$ in place of 
$\chi_{]s, T[}\vert_{]0, T[}$.

Let $\kappa\in C^0\big([0, T], B_{p, q}^{-1/p'+\eta}
(]0, T[, \mathbb{R})\big)$. Our goal is to give a meaning
to the formula $\mathcal{H}_T(u)(t) = 
\int_0^tu(s)\kappa(s, t)ds$
for $u\in L^{\infty}(]0, T[, \mathbb{R})$. 
Let 
$\phi \in 
B_{p', q'}^{1/p'-\eta}(]0, T[, \mathbb{R})
$
and $s\in]0, T[$. We have
\begin{align}
\vert
<\kappa(s), \chi_{]s, T[}\phi>
\vert
&\leq
C_T \Vert
\kappa(s)
\Vert_{B_{p, q}^{-1/p'+\eta}(]0, T[)}
\Vert
\phi
\Vert_{B_{p', q'}^{1/p'-\eta}(]0, T[)}\label{xhk}
\end{align}
using the continuity of the bracket and the fact 
that 
$\chi_{]s, T[}$ is a multiplier for
$B_{p', q'}^{1/p'-\eta}(]0, T[)$. Hence
\begin{align}
 \vert
 \int_0^T u(s)<\kappa(s), 
 \chi_{]s, T[}\phi>ds
 \vert
 \leq C_T \Vert u \Vert_{L^{\infty}}
 \Vert \kappa
 \Vert_{C^0([0, T], B_{p, q}^{-1/p'+\eta})}
  \Vert \phi
 \Vert_{B_{p', q'}^{1/p'-\eta}}\label{narf}
\end{align}
By theorem 11 in \cite{mu} and the fact that 
$B_{p', q', [0, T]}^{1/p'-\eta}(\mathbb{R}) = 
B_{p', q'}^{1/p'-\eta}(]0, T[)$ within the range of
indexes $-1/p < 1/p' - \eta < 1/p'$, 
we conclude that 
$$
\mathcal{H}_T:
L^{\infty}(]0, T[, \mathbb{R})
\rightarrow
B_{p, q}^{-1/p'+\eta}(]0, T[, \mathbb{R})
$$
with
$
<\mathcal{H}_T(u), \phi>
= 
\int_0^T
u(s)<\kappa(s), 
 \chi_{]s, T[}\phi>ds
$ for any 
$\phi\in B_{p', q'}^{1/p'-\eta}(]0, T[, 
\mathbb{R})$,
is a well defined and continuous operator. Restricted to
$B_{p, q}^{1/p+\alpha}(]0, T[, \mathbb{R})$, 
it satisfies 
properties $L1$ and $L2$. 
Due to theorem $\ref{terminable}$ b), we thus have
the existence and uniqueness of a solution
$u\in B_{p, q}^{1/p+\eta}(]0, T[, 
\mathbb{R})$
of problem $u'(t) = \int_0^t u(s)\kappa(s, t)ds$ with 
$u(0) = u_0$. 

This can be generalized to related
operators, for instance 
$\int_0^t f(u(s))\kappa(s, t)ds$, in the scalar or 
Banach-valued case; or operators 
$\int_0^t f\big(m(u)(s)\big)\kappa(s, t)ds$ with 
$m(u)(s) = sup_{0\leq \sigma < s}1/(s-\sigma)
\int_{\sigma}^s
\vert u(z)\vert dz$
etc...

d) Let $\mathcal{A}$ be a Banach algebra. 
Let $sup\{-1/p', -1/2\}< \sigma < 1/p$. 
Let $\psi_j\in B_{p, q}^{\sigma}(]0, T[, 
\mathcal{A})$ and let $f_j: \mathcal{A}\rightarrow
\mathcal{A}$ be given Lipschitz functions
$(j\in \mathbb{Z})$. Consider
$$
\mathcal{H}_T:B_{p, q}^{s}(]0, T[, 
\mathcal{A})
\rightarrow
B_{p, q}^{\sigma}(]0, T[, 
\mathcal{A})
$$
with 
\begin{align}
\mathcal{H}_T(u) = \sum_{j\in \mathbb{Z}}
f_j(u)\psi_j\label{nouille}
\end{align}
(for results on composition operators, see 
\cite{BS}). Due to the hypothesis on $\sigma$, 
one can find $s$ with 
$1/p< s< \sigma+1< 1/p+1$
and $s+\sigma > 0$. For such $s$, 
 lemma $\ref{para}$ applies and the product is continuous in
$B_{p, q}^{s}(]0, T[, 
\mathcal{A})
\times
B_{p, q}^{\sigma}(]0, T[, 
\mathcal{A})
\rightarrow 
B_{p, q}^{s}(]0, T[, 
\mathcal{A})$. Hence, 
$
\mathcal{H}_T$
is well defined, continuous and satisfies
properties L1 and L2 under the (sufficient)
assumption that, for any 
$(u, v) \in 
B_{p, q}^{s}(]0, T[, 
\mathcal{A})^2$ with $u \neq v$, we have
\begin{align}
\sum_{j\in \mathbb{Z}}
\Vert \kappa_j
\Vert_{B_{p, q}^{\sigma}}
\big(
\Vert
f_j(v)-f_j(u)
\Vert_{B_{p, q}^{s}}
/
\Vert
v-u
\Vert_{B_{p, q}^{s}}
+
\Vert
f_j(0)
\Vert_{B_{p, q}^{s}}
\big)
\leq M < \infty\label{7extra}
\end{align}
Definition
$\ref{nouille}$ can be generalized for instance as
 $$\mathcal{H}_{T}(u) = 
 \int_{\mathbb{R}^{n}}
\phi(s, u) \kappa(s) d\mu(s)$$
with $\kappa \in C^0\big(\mathbb{R}^{n}, 
B^{\sigma}_{p,  q}(]0, 
T[, \mathcal{A})\big)$, 
$\phi\in C^0\big(\mathbb{R}^{n}
\times 
B^
{s}_{p, q}
(]0, T[, \mathcal{A}), 
B^
{s}_{p, q}(]0, 
T[, \mathcal{A})\big)$,
$\mu$ a Radon measure on $\mathbb{R}^{\uu{n}}$ and 
with an extra 
condition analogous to $\ref{7extra}$.

$\bf{Acknowledgements}$. This work was
supported by the LABEX MILYON 
(ANR-10-LABX-0070) of University 
de Lyon, within the program 
"Investissement d'Avenir" (ANR-11-IDEX 
0007) operated by the French National 
Research Agency (ANR).

\end{document}